\documentclass[reqno]{amsart}
\usepackage[scale=0.75, centering, headheight=14pt]{geometry}
\usepackage[latin1]{inputenc}
\usepackage[T1]{fontenc}
\usepackage{lmodern}
\usepackage[english]{babel}
\usepackage{esint}

\usepackage{tikz}
\usepackage{bbm}
\usepackage{amsmath,amssymb,amsfonts,amsthm}
\usepackage{mathtools,accents}
\usepackage{mathrsfs}
\usepackage{xfrac}
\usepackage{array} 
\usepackage{aliascnt}
\usepackage{booktabs} 
\usepackage{array} 

\usepackage{verbatim} 
\usepackage{subfig} 

\usepackage{mathrsfs, dsfont}
\usepackage{amssymb}
\usepackage{amsthm}
\usepackage{amsmath,amsfonts,amssymb,esint}
\usepackage{graphics,color}
\usepackage{enumerate}
\usepackage{mathtools,centernot}

\usepackage{microtype}
\usepackage{paralist} 
\usepackage{cases}
\usepackage[initials]{amsrefs}
\allowdisplaybreaks

\usepackage{braket}
\usepackage{bm}

\usepackage[citecolor=blue,colorlinks]{hyperref}
\addto\extrasenglish{}

\usepackage{enumerate}
\usepackage{xcolor}
\usepackage{aliascnt}

\makeatletter
\def\newaliasedtheorem#1[#2]#3{
  \newaliascnt{#1@alt}{#2}
  \newtheorem{#1}[#1@alt]{#3}
  \expandafter\newcommand\csname #1@altname\endcsname{#3}
}
\makeatother

\theoremstyle{plain}
\newtheorem{theorem}{Theorem}[section]
\newaliasedtheorem{lemma}[theorem]{Lemma}
\newaliasedtheorem{prop}[theorem]{Proposition}
\newaliasedtheorem{claim}[theorem]{Claim}
\newaliasedtheorem{corollary}[theorem]{Corollary}
\theoremstyle{remark}

\theoremstyle{definition}
\newaliasedtheorem{definition}[theorem]{Definition}
\newaliasedtheorem{example}[theorem]{Example}

\theoremstyle{remark}
\newaliasedtheorem{remark}[theorem]{Remark}

\numberwithin{equation}{section}

\def\R{\mathbb R}
\def\N{{\mathbb N}}
\def\Z{{\mathbb Z}}
\def\T{{\mathbb T}}
\def\Q{{\mathbb Q}}

\def\nop{12}

\DeclareMathOperator{\diver}{div}

\newcommand*{\RR}{\ensuremath{\mathcal{R}}}

\title{Non-uniqueness of integral curves for autonomous hamiltonian vector fields}

\address{Vikram Giri
\hfill\break  Princeton University, Department of Mathematics,
Washington Rd., Princeton, NJ 08544, USA}
\email{vgiri@math.princeton.edu}

\author[V. Giri and M. Sorella]{Vikram Giri and Massimo Sorella}
\address{Massimo Sorella
\hfill\break  \'Ecole Polytechnique F\'ed\'erale de Lausanne, Institute of Mathematics, Station 8, CH-1015 Lausanne, Switzerland.}
\email{massimo.sorella@epfl.ch}

\begin{document}

\maketitle

\begin{abstract}
In this work we prove the existence of an autonomous Hamiltonian vector field in $W^{1,r}(\T^d; \R^d)$ with $r< d-1$ and $d \geq 4$ for which the associated transport equation has non-unique positive solutions. As a consequence of Ambrosio's superposition principle \cite{A04}, we show that this vector field has non-unique integral curves with a positive Lebesgue measure set of initial data and moreover we show that the Hamiltonian is not constant along these integral curves. 
\end{abstract}
\par
\medskip\noindent
{\sc MSC (2020): 70H33 - 35A02 - 35D30 - 35Q49 - 34A12.
\par
}

\section{Introduction}\label{sec:intro}

This paper is concerned with the study of Hamiltonian system which are deeply studied because they model various physical systems. More precisely, we consider the following  ODEs on the $d=2d'$ dimensional torus $\T^d \simeq \R^d/ \Z^d$


\begin{align}\label{ode}
\begin{cases}
\dot{\gamma} (t) = J \nabla H(t, \gamma(t))
\\
\gamma (0) = x_0,
\end{cases}
\end{align}
where $H : \T^d \rightarrow \R$ is the autonomous Hamiltonian function, $x_0 \in \R^d$, $\gamma : [0,1] \rightarrow \T^d$ is a curve and $J $ is the $d \times d$ matrix
 $$J=\begin{pmatrix}
 0_{d'}  & I_{d'} \\
 -I_{d'} & 0_{d'}
 \end{pmatrix}.$$
 
  Such a $\gamma$ is called integral curve of the {\em Hamiltonian  vector field} $J \nabla H$ starting from $x_0$, more precisely we give the following definition.
\begin{definition}\label{defn:int-curve}
Let  $u: (0,1) \times \T^d \to \R^d$ be a Borel map. We say that $\gamma \in AC([0,1]; \T^d)$ is an integral curve of $u$ starting at $x$ if $\gamma(0)=x$ and $\gamma'(t) = u(t, \gamma(t))$ for a.e. $t\in [0,1]$.
\end{definition}  
With such definition we recall that if the vector field is Lipschitz, which in our case means that $H \in W^{2, \infty}(\T^d, \R)$, then by Cauchy-Lipschitz theorem we have uniqueness of integral curves of \eqref{ode} for every starting point $x_0 \in \T^d$. 

In the regular setting, i.e. when the Hamiltonian is in $W^{2, \infty}(\T^d)$, we have that the integral curves characterize the unique solution of the following continuity equation 
\begin{align} \label{d_continuity}
\begin{cases}
\partial_t \rho + \diver_x (\rho J \nabla H) = 0,
\\
\rho(\cdot, 0)= \rho_0(\cdot),
\end{cases}
\end{align}
where the unknown is the density $\rho: [0,1] \times \T^d \rightarrow \R$, while the initial datum $\rho_0 \in L^\infty (\T^d)$ and the vector field $J \nabla H \in W^{1, \infty}(\T^d, \R^d)$ are given. More precisely, in this setting the unique solution (uniqueness holds thanks to a simple application of Gr\"onwall lemma) of \eqref{d_continuity}  is given by\footnote{We denote with $X(t, \cdot)_\# (\mu) $ the pushforward of the measure $\mu$ through the map $X$ at the fixed time $t$. } 

\begin{equation}\label{RLF_continuity}
\rho(t, \cdot ) \mathscr{L}^d=X(t,\cdot )_\# ( \rho(0, \cdot ) \mathscr{L}^d),
\end{equation}

where $X : [0,1] \times \T^d \to \R^d$ is  called the flow related to $J \nabla H$ and is defined collecting all the integral curves, namely $X(t, x_0)= \gamma_{x_0}(t)$, where $\gamma_{x_0}$ is the unique integral curve of \eqref{ode} starting from $x_0$.  The last result is classical and is known in the literature as Liouville's theorem.

 A lot of efforts were made in the last decades to understand what happens in the non regular setting. 
The first classical failing is the nonuniqueness of integral curves with vector fields $ u \in C^{0, \alpha} (\T^d, \R^d)$ for $\alpha \in [0,1)$, that means that the previous results cannot be extended easily to less regular spaces.

From now on we will refer to solutions of \eqref{d_continuity} in the sense of distribution, so the reasonable assumption to talk about that is that $H \in W^{1,q}(\T^d, \R^d)$ and $\rho \in L^1 ((0,1), L^p(\T^d))$.

 We now summarize the main results of the last decades that are important to understand our contribution. To do so we will present all the results in the particular case of autonomous hamiltonian vector fields, even if they do not require neither the autonomous hypothesis nor the hamiltonian.  The interested reader can find the general results in the references we provide.
 
DiPerna and Lions proved in \cite{DPL89} a result that implies the uniqueness of \eqref{d_continuity} solutions in the class of densities $\rho \in L^{\infty}((0,1), L^p(\T^d))$ once given $\rho_0 \in L^{p}(\T^d)$ and the autonomous hamiltonian vector field $H \in W^{2, q}(\T^d)$, with $p$ and $q$ such that

\begin{equation}
 \frac{1}{p} + \frac{1}{q} \leq 1.
\end{equation}

The authors in \cite{DPL89} consider also a selection of integral curves, which defines  the so-called ``Regular Lagrangian flow''. In turn, this characterizes the unique solution $\rho$ by \eqref{RLF_continuity} using as flow the ``Regular Lagrangian flow''.
To be precise here we give the definition of the ``Regular  Lagrangian flow'' with the compressibility condition following Ambrosio's definition (see in \cite[Section 6]{A04}).

\begin{definition}[Regular Lagrangian flow]\label{RLF}
Let $u:(0,1)\times\T^d\to \R^d$ be Borel. We say that a Borel map $X:[0,1]\times \R^d\to \R^d$ is a regular Lagrangian flow of $u$ if 
	\begin{enumerate}[(i)]
		\item for $\mathscr{L}^d$-a.e. $x\in \T^d$, $t\mapsto X(t,x)$ is integral curve of $u$ with $X(0,x)=x$,
		\item there is a constant $C>0$ such that for every $t\in[0,1]$, $X(t,.)_\#\mathscr{L}^d\leq C\mathscr{L}^d$.
	\end{enumerate}
\end{definition}


In \cite{A08} Ambrosio proved the superposition principle, which implies the following: every {\em non negative } solution $\rho \in L^1$ to \eqref{d_continuity} with an autonomous hamiltonian vector field such that  $ \rho J \nabla H \in L^{1}$ are transported by a ``generalized flow'', which is roughly speaking a measure supported on integral curves of the associated vector field. See \cite[Theorem 3.2]{A08} for the precise result, known in the literature as ``Ambrosio's superposition principle''. This result connects the ODE \eqref{ode} with the PDE \eqref{d_continuity} in a larger class of regularity where uniqueness of solutions of \eqref{d_continuity}  fails and hence it is crucial for us (see in \cite{MoSz2019AnnPDE,MS20,MoSz2019Calc,BDLC20}).

We now turn to explain our contribution in this context. Large part of this paper is dedicated to prove a nonuniqueness result for positive solutions of the continuity equation \eqref{d_continuity} for Hamiltonian autonomous vector fields, more precisely we will prove the following theorem.

\begin{theorem}\label{t_ce}
Let 
${{d\geq 4}}$ be an even integer, $p \in (1,\infty), r \in [1,\infty]$ be such that 
\begin{equation}
\label{hp:exp}
\frac 1p + \frac 1r > 1+ \frac 1{d-1}
\end{equation}
and denote by $p'$ the dual exponent of $p$, i.e. $\frac{1}{p}+\frac{1}{p'}=1$. 
Then for every $T>0$ there exists an autonomous Hamiltonian $H \in W^{2,r}(\T^d; \R)$ and a  nonconstant $\rho \in C ([0,T], L^1 (\T^d))$ such that $\rho J \nabla H \in C ([0,T], L^1 (\T^d))$ and \eqref{d_continuity} holds with initial data $\rho(0, \cdot) =1$ and for which $\rho \geq c_0$ for some positive constant $c_0$. Moreover, if $p' \geq d-1$, we have the Hamiltonian $H \in C(\T^d; \R)$.
\end{theorem}

The theorem above is proved using the ``convex integration type'' techniques borrowed from a groundbreaking work of Modena and Sz\'ekelyhidi \cites{MoSz2019AnnPDE,MoSz2019Calc} and subsequently improved  by Modena and Sattig \cite{MS20}. 
 We refer to \cite{DLSZ09,DLSZ13,DLSZ17,Is18,BV} and the references therein for the birth of this and related lines of research. See also some  recent results related to convex integration in \cite{CDRS21,CL20,CL20NS,CL21,SP21}.
 

As a consequence of Ambrosio's superposition principle and Theorem \ref{t_ce} we prove a non uniqueness result of integral curves for autonomous Hamiltonian vector fields. This strategy has already been used in the work of Bru\'e, Colombo and De Lellis \cite[Theorem 1.3]{BDLC20}.

\begin{theorem} \label{t_main}
For any even integer $d \geq 4$ and any real number $r< d-1$  there is an autonomous  Hamiltonian  $H \in C(\T^d; \R) \cap W^{2,r}(\T^d; \R)$ such that the following holds for every Borel map $v$ with $v= J \nabla H \ \mathcal{L}^{d}-a.e.:$

(NU) 
There is a measurable set $A \subset \T^d$ with positive Lebesgue measure such that for every $x \in A$ there

\hspace{0.8cm}  are at least two integral curves of $v$ starting at $x$.
\end{theorem}

The case $d=2$ is not included in our theorem: indeed in the 2 dimensional setting our theorem statement would tell us that $\frac{1}{p} + \frac{1}{r} > 2$ which is impossible for $p,r \geq 1$. Moreover, in \cite[Theorem 5.2]{ABC09}, the authors proved a uniqueness result  in dimension $2$ which implies the uniqueness for the continuity equation \eqref{d_continuity} in the class $\rho \in L^1((0,1) \times \R^2)$ for autonomous bounded Hamiltonian vector fields such that $H \in W^{2,1} \cap W^{1, \infty}$ (see \cite[Section 2.15 (iii)]{ABC09} for this implication).
They also explain in \cite[Section 6.2]{ABC09} the additional assumption that is needed in dimensions $d>2$, which in turn is necessary and in general not satisfied by $W^{2,r}(\T^d)$ with $r<d-1$, because of our result Theorem \ref{t_main}. 
 We highlight that the Sard property, which is a key ingredient in \cite[Theorem 5.2]{ABC09},  is true in Sobolev spaces $W^{d,1}$ without any Lipschitz assumption thanks to \cite{BKK15}, where $d$ is the dimension of the space.


Finally we discuss the ``conservation of the Hamiltonian'' along integral curves. It is well known that for a smooth Hamiltonian vector field, the Hamiltonian is constant along integral curves of the flow.  Indeed, in physics  it represents the total energy of the system. For Sobolev fields, the same conclusion is true for the integral curves of the associated Regular Lagrangian Flow. This can be proved by an approximation argument. However, the Hamiltonian is not necessarily constant along the non-unique integral curves constructed in Theorem \ref{t_main}; this is the content of the next theorem.

\begin{theorem}\label{thm:hamilnotcons}
For any integer $d \geq 4$, and any real number $r < d-1$, there is an autonomous Hamiltonian  $H \in C(\T^d; \R) \cap W^{2,r}(\T^d; \R)$ such that the following holds for every $v = J \nabla H$ $\mathcal{L}^d- $a.e.. There is a measurable set $A \subset \T^d$ with positive Lebesgue measure such that for every $x \in A$ there exists $\gamma_x$ integral curve of $v$ starting from $x$ and
 $$H (\gamma_x(0)) > H (\gamma_x(1)),$$
in particular $H$ is not conserved along some integral curves starting from a non negligible set.
\end{theorem}

We now highlight the main new technical ideas in the convex integration scheme used to prove Theorem \ref{t_ce}:
\begin{itemize}
    \item We perturb the autonomous vector field in a ``universal way'', i.e. independently from the previous error (that in the literature is called ``Reynolds error''). This allows us to preserve the autonomous property of the Hamiltonian.
    \item We notice that it is not necessary to have $\rho \in L^p$ and $J \nabla H \in L^{p'}$, but only need that $\rho, \rho J \nabla H  \in L^1$ in order to make sense of a distributional solution to \eqref{d_continuity} and to apply Ambrosio superposition principle.
    \item We construct autonomous Hamiltonian vector fields (used in the scheme to perturb the previous Hamiltonian vector field)\footnote{This functions used in the perturbation step are called in the literature ``building blocks''.}.
\end{itemize}

\section{Preliminary Lemmas}

\subsection{Geometric lemma} We start with an elementary geometric fact, namely that every vector in $\mathbb R^d$ can be written as a positive linear combination of elements in a suitably chosen finite subset $\Lambda$ of $\mathbb{Q}^d\cap \partial B_1$. This is reminiscent of the geometric lemma in \cite{DLS13} and it is proved in \cite[Lemma 3.1]{BDLC20}.

\begin{lemma}\label{lemma:geom}
	There exists a finite set $\{ \xi\}_{\xi \in \Lambda} \subseteq \partial B_1\cap \mathbb{Q}^d$ and smooth non-negative coefficients $a_\xi(R)$ such that for every $R \in \partial B_1$ 
	\[
	R= \sum_{\xi \in \Lambda} a_\xi(R) \xi\, .
	\]
	Moreover, for each $\xi \in \Lambda$ there exists $\xi^\perp, \xi_1, .., \xi_{d-2}$ such that $\{ \xi, \xi^\perp, \xi_1, ..,\xi_{d-2} \} \subset \partial B_1 \cap \Q^d $ form an orthonormal  basis of $\R^d$ and $J \xi^\perp = \xi$.
	Finally, since we will periodize functions, let $n_\ast \in \N$  be
$$ \max_{\xi \in \Lambda} l_\xi, $$
where $l_{\xi}$ is	
	  the l.c.m. of the denominators of the rational numbers $\xi, \xi^\perp, \xi_1,..,\xi_{d-2}$. 
\end{lemma}

And we also recall the following result from \cite[Lemma 4.2]{BDLC20} to get the property of building blocks disjoint supports.

\begin{lemma}\label{lemma:disjointsupports}
Let $d \geq 3$, $\frac14 >\rho>0$ and $\Lambda \subseteq \mathbb S^{d-1} \cap \mathbb Q^d$ be a finite number of vectors. Then there exists $\mu_0:= \mu_0 (d , \Lambda) >0$ and a family of vectors $\{v_\xi\}_{\xi \in \Lambda} \subseteq \R^d$ such that the periodized cylinders 
$v_\xi+ B_{2\rho \mu^{-1}} + \mathbb R \xi + \mathbb Z^d
$ are disjoint as $\xi$ varies in $\Lambda$, provided $\mu \geq \mu_0$. 
\end{lemma}

\subsection{Antidivergences}
We recall that the operator $\nabla \Delta^{-1}$ is an anti-divergence when applied to smooth vector fields of $0$ mean. As shown in \cite[Lemma 2.3]{MoSz2019AnnPDE} and \cite[Lemma 3.5]{MS20}, however, the following lemma introduces an improved anti-divergence operator, for functions with a particular structure.

\begin{lemma}\label{lemma23}(Cp. with \cite[Lemma 3.5]{MS20})
	Let $\lambda \in \N$ and $f, g : \T^d \to \R$ be smooth functions, and $g_\lambda= g(\lambda x)$. 
	Assume that $\int g = 0$. Then if we set $\RR (f g_\lambda) = f \nabla \Delta^{-1} g_\lambda -\nabla \Delta^{-1} (\nabla f \cdot \nabla \Delta^{-1}g_\lambda+\int fg_\lambda)$, we have that $\diver  \RR (f g_\lambda) = f g_\lambda-\int fg_\lambda$ and for some $C:=C({k,p})$
	\begin{equation}
	\label{ts:antidiv}
	\|D^k \RR (f g_\lambda)\|_{L^p} \leq C \lambda^{k-1} \|f\|_{C^{k+1}} \| g\|_{W^{k,p}} \qquad \mbox{for every } k\in \N, p\in [1,\infty].
	\end{equation}
\end{lemma}
\begin{proof}
	It is enough to combine \cite[Lemma 3.5]{MS20} and the remark in \cite[page 12]{MS20}.
\end{proof}

\subsection{Slow and fast variables}
Finally we recall the following improved H\"older inequality, stated as in \cite[Lemma 2.6]{MoSz2019AnnPDE} (see also \cite[Lemma 3.7]{BV}).
If $\lambda \in \N$ and $f,g:\T^d \to \R$ are smooth functions, then we have 
\begin{equation}
\label{eqn:impr-holder}
\| f(x) g(\lambda x) \|_{L^p} \leq \| f \|_{L^p} \| g \|_{L^p} + \frac{C(p)\sqrt d \|f \|_{C^1} \|g\|_{L^p}}{\lambda^{1/p}}
\end{equation}
and 
\begin{equation}
\label{eqn:l26}
\Big| \int f(x) g(\lambda x) \, dx \Big| \leq\Big| \int f(x) \Big(g(\lambda x) - \int g \Big) \, dx \Big| + \Big| \int f \Big |\cdot \Big| \int g \Big| \leq \frac{\sqrt d \|f \|_{C^1} \|g\|_{L^1}}{\lambda} + \Big| \int f \Big| \cdot \Big | \int g \Big| .
\end{equation}

\section{Building blocks}
Let $0<\rho<\frac 14$ be a constant
. We consider $\varphi\in C^\infty_c (\R^{d-1})$ and $\psi\in C^\infty_c (\R^{d-1})$ which satisfy
	$$\varphi \in C^{\infty}_c(B_\rho) \qquad \int \varphi =1, \qquad  \varphi \geq 0, \qquad 
	$$
	and 
	$$\psi \in C^\infty_c(B_{2\rho}), \qquad \int \psi=0, \qquad \psi (x_1,x_2,..,x_{d-1}) = x_1  \mbox{ on }B_\rho.$$
Given $\mu \gg 1$ we define 
\begin{align*}
\overline{\varphi}_\mu(x):= \mu^{(d-1)/p} \varphi(\mu x),
\\
\overline{\psi}_\mu(x):= \mu^{(d-1)/p'} \psi(\mu x).
\end{align*}

By an abuse of notation, we periodize $ \overline{\varphi} , \overline{\psi}_\mu $ so that the functions are treated as periodic functions defined on $\T^{d-1}$. These periodic functions will allow us to define our building blocks, defined on $\T^d$.

Given $\Lambda$ and $n_\ast \in \N$ as in Lemma \ref{lemma:geom}, for any $\xi \in \Lambda$ (we recall the notation of $\xi^\perp, \xi_1,..,\xi_{d-2}$ as in Lemma \ref{lemma:geom}) and for any $\sigma >0$, we define $\tilde{\Theta}_{\xi, \mu, \sigma}, \tilde{H}_{\xi,\mu} : \T^d \to \R$ and the autonomous Hamiltonian vector field $X_{\tilde{H}_{\xi,\mu}} : \T^d \to \R^d$

\begin{align*}
{\Theta}_{\xi,\mu, \sigma} (x) := \sigma n_\ast^{d-1} \overline{\varphi}_\mu (n_\ast \xi^\perp \cdot (x - v_{\xi}), n_\ast \xi_1 \cdot (x - v_{\xi}), n_\ast \xi_2 \cdot (x - v_{\xi}), .., n_\ast \xi_{d-2}  \cdot (x - v_{\xi})),
\\
{H}_{\xi,\mu}(x) := \frac{1}{\mu}   \overline{\psi}_\mu (n_\ast \xi^\perp \cdot (x - v_{\xi}), n_\ast \xi_1 \cdot (x - v_{\xi}), n_\ast \xi_2 \cdot (x - v_{\xi}), .., n_\ast \xi_{d-2}  \cdot (x - v_{\xi})),
\\
X_{{H}_{\xi,\mu}} := J \nabla {H}_{\xi,\mu}(x).
\end{align*}
where  $\mu \geq \mu_0(d, \Lambda)$ and  $\{v_{\xi} \}_{\xi \in \Lambda}$ are given by Lemma \ref{lemma:disjointsupports} in order to get the following property on the supports of these family of functions
$$\text{supp} X_{H_{\xi, \mu}}   \cap \text{supp} \Theta_{\xi',\mu,\sigma} = \text{supp} X_{H_{\xi, \mu}} \cap \text{supp} X_{H_{\xi', \mu}} = \text{supp} \Theta_{\xi',\mu,\sigma} \cap \text{supp} \Theta_{\xi,\mu,\sigma} = \emptyset,$$
for any $\xi \neq \xi' \in \Lambda$.
\begin{remark}
In the previous definition we just multiply $\varphi$ by $\sigma$, because we will need autonomous vector fields, but the natural choice is to split $\sigma>0$ as $\sigma^{1/p}$ and $\sigma^{1/p'}$ as in \cite[Section 4]{BDLC20}.
\end{remark}

Finally, by standard computations, we have proved the following fundamental lemma for our building blocks.

\begin{lemma} \label{lemma:buildingblocks}
Let $d \geq 4$, $\Lambda \subset \partial B_1 \cap \Q^d$ be a finite set. Then there exists $\mu_0 >0$ such that the following holds.

 There exist two families of functions
$\{ \Theta_{\xi, \mu,\sigma} \}_{\xi, \mu, \sigma} \subset C^\infty(\T^d)$, $\{ H_{\xi, \mu,} \}_{\xi, \mu} \subset C^\infty(\T^d;\R^d)$, where $\xi \in \Lambda$ and $ \sigma, \mu \in \R$ such that 
 for any  $\mu \geq \mu_0$, $\sigma >0$  we have
 
 \begin{equation}\label{eqn:itsolves}
	\diver ( X_{H_{\xi, \mu}} \Theta_{\xi, \mu, \sigma})=0,
	\end{equation}
 $$\diver X_{H_{\xi, \mu}}=0,$$ 
 \begin{equation}
 \label{eqn:avW}
 \int  X_{H_{\xi, \mu}}=0,
 \end{equation}
	\begin{equation}\label{eqn:rightaverage}
	\int  X_{H_{\xi, \mu}} \Theta_{\xi, \mu, \sigma}= \sigma \xi.
	\end{equation}
For any $k\in \N$ and any $s\in [1,\infty]$ one has 
\begin{equation}\label{eqn:Thetanorms}
\| D^k \Theta_{\xi, \mu, \sigma} \|_{L^s}\le C(d,k,s, n_\ast) \sigma \mu^{k + (d-1) ( 1/p - 1/s )},
\end{equation}
\begin{equation}\label{eqn:Hamiltonian_norms}
\| D^k  X_{H_{\xi, \mu}} \|_{L^s} \le C(d,k,s, n_\ast) \mu^{ k+ (d-1)(1/p'-1/s)},
\qquad \| D^k  H_{\xi, \mu} \|_{L^s} \le C(d,k,s, n_\ast) \mu^{ k-1+ (d-1)(1/p'-1/s)}.
\end{equation}
Finally, they have pairwise compact disjoint supports for any $\xi \neq \xi'$, namely
\begin{align} \label{eqn:disjointsupp}
\text{supp} X_{H_{\xi, \mu}}   \cap \text{supp} \Theta_{\xi',\mu,\sigma} = \text{supp} X_{H_{\xi, \mu}} \cap \text{supp} X_{H_{\xi', \mu}} = \text{supp} \Theta_{\xi',\mu,\sigma} \cap \text{supp} \Theta_{\xi,\mu,\sigma} = \emptyset,
\end{align}
for any $\xi \neq \xi'$.
\end{lemma}

\section{Iteration scheme}
As in \cite{MoSz2019AnnPDE} we consider the following system of equations in $[0,T] \times \T^d$
\begin{equation}\label{eqn:CE-R}
\partial_t \rho_{q} + \diver (  \rho_{q} J \nabla H_{q})=-\diver R_{q},
\end{equation}
where we observe that  $\diver J \nabla H =0. $
We then fix three parameters $a_0$, $b>0$ and $\beta>0$, to be chosen later only in terms of $d$, $p$, $r$, and for any choice of $a>a_0$ we define
\[
\lambda_0 =a, \quad \lambda_{q+1} = \lambda_q^b \quad \mbox{and}\quad
\delta_{q} = \lambda_{q}^{-2\beta}\, .
\]
The following proposition builds a converging sequence of functions with the inductive estimates
\begin{equation}
\label{eqn:ie-1}
\max_t \|R_q (t, \cdot)\|_{L^1} \leq \delta_{q+1}
\end{equation}
\begin{equation}
\label{eqn:ie-2}
\max_t \left(\| \rho_q (t, \cdot)\|_{C^1} +\|\partial_t \rho_q (t, \cdot)\|_{C^0} + \|H_q \|_{C^1} + \|H_q \|_{W^{2,p'}} + \| H_q \|_{W^{3,r}} \right) 
\leq \lambda_{q}^\alpha\, ,
\end{equation}
where $\alpha$ is yet another positive parameter which will be specified later. 

\begin{prop}\label{prop:inductive}
	There exist $\alpha,b, a_0, M>5$, $0<\beta<(2b)^{-1}$ such that the following holds. For every $a\geq a_0$, if
	$(\rho_q, H_q, R_q)$ solves \eqref{eqn:CE-R} and enjoys the estimates \eqref{eqn:ie-1}, \eqref{eqn:ie-2}, then there exist $(\rho_{q+1}, H_{q+1}, R_{q+1})$ which solves \eqref{eqn:CE-R}, enjoys the estimates \eqref{eqn:ie-1}, \eqref{eqn:ie-2} with $q$ replaced by $q+1$ and also the following properties:
	\begin{itemize}
		\item[(a)] $\max_t \|( \rho_{q+1}- \rho_q) (t, \cdot)\|_{L^1}  \leq  \delta_{q+1}$
		\item [(b)] $\|H_{q+1}- H_q \|_{W^{2,r}}  + \|H_{q+1}- H_q\|_{W^{1,p'}} \leq  \frac{M}{2^q}$,
		\item [(b+)] If $p' \geq  d-1$, then $\|H_{q+1}- H_q \|_{L^\infty} \leq \frac{M}{ \lambda_{q}} $
		\item [(c)] $\max_t \| \rho_{q+1} J \nabla H_{q+1} - \rho_q J \nabla H_q \|_{L^1} \leq M \delta_{q+1}$
		\item[(d)] $\inf (\rho_{q+1} - \rho_q) \geq - \delta_{q+1}$
		\item[(e)] if for some $t_0>0$ we have that $\rho_q(t, \cdot) = 1$ and $R_q(t, \cdot)=0$ for every $t\in [0,t_0]$, then $\rho_{q+1}(t, \cdot) = 1$ and $R_{q+1}(t, \cdot)=0$  for every $t\in [0,t_0- \lambda_q^{-1-\alpha}]$.
	\end{itemize}
\end{prop}
Our iterative proposition is quite similar to the one proposed in \cite{BDLC20}, but in order to get an autonomous vector field, we put all the `bad' terms form the previous Reynolds error onto the definition of the new density. This gives much worse estimates on the new density: indeed we are unable to control its $L^p$ norm. But the concentration parameter $\mu$ allows us to control the density in the $L^1$ norm. Point (c) in the above proposition allows us to show convergence of $\rho_q J\nabla H_q$ in $L^1$ which was previously done in \cite{BDLC20} using H\"older's inequality.

Also, in order to get an autonomous vector field, we are unable to use the non-autonomous building blocks of \cite{BDLC20} but instead use the autonomous `Mikado flows' that were used as building blocks in \cite{MoSz2019AnnPDE}. The disadvantage of this approach that was already apparent in \cite{MoSz2019AnnPDE}, is that we are unable to get the full dimensional concentration $d$ but are only able to get a $d-1$ in eqn. \ref{hp:exp}.
\subsection{Choice of the parameters}\label{sec:param}
The choice of parameter is very similar to those in \cite[Section 5.1]{BDLC20}.
We define first the constant
\[
\gamma:= \Big(1+\frac 1p\Big) \left(\min \Big\{ \frac{d-1}{p}, \frac{d-1}{p'}, -1-(d-1)  \Big(\frac1{p'}-\frac1r\Big)\Big\}\right)^{-1}>0,
\]
Notice that, up to enlarging $r$, we can assume that the quantity in the previous line is less than $1/2$, namely that $\gamma>2$.
Hence we set $\alpha:=4 + \gamma (d+1)$,
\begin{equation}
\label{eqn:choice-b}
b := \max\{ p, p'\} (3(1+\alpha)(d+2)+2),
\end{equation}
and 
\begin{equation}
\label{eqn:choice-beta}
\beta:= \frac 1{2 b} \min\Big\{ p, p', r, \frac{1}{b+1}\Big\}= \frac{1}{2 b(b+1)} .
\end{equation}
Finally, we choose $a_0$ and $M$ sufficiently large (possibly depending on all previously fixed parameters) to absorb numerical constants in the inequalities.
We set
\begin{equation}
\label{eqn:choice-ell}
\ell: = \lambda_{q}^{-1-\alpha},
\end{equation}
\begin{equation}
\label{eqn:choice-mu}\mu_{q+1} := \lambda_{q+1}^{\gamma}.
\end{equation}

\subsection{Convolution}
The convolution step is the same of \cite[Section 5.2]{BDLC20}. 
We first perform a convolution of $\rho_q$ and $u_q$ to have estimates on more than one derivative of these objects and of the corresponding error. Let $\phi \in C^\infty_c(B_1)$ be a standard convolution kernel in space-time, $\ell$ as in \eqref{eqn:choice-ell} and define 
$$\rho_{\ell} := \rho_{q} \ast \phi_\ell, \qquad H_\ell := H_q \ast \phi_\ell, \qquad u_\ell := u_q \ast \phi_\ell, \qquad R_{\ell} := R_{q} \ast \phi_\ell. $$
We observe that $(\rho_{\ell}, u_\ell, R_{\ell}+ (\rho_{q} u_q)_\ell - \rho_{\ell} u_\ell)$ solves system \eqref{eqn:CE-R}
and  by \eqref{eqn:ie-1}, \eqref{eqn:choice-beta} enjoys the following estimates

\begin{equation}
\label{eqn:r-l-l1}
\|R_\ell \|_{L^1} \leq \delta_{q+1},
\end{equation}
\begin{equation} \label{eqn:rho-ell-conv}
\| \rho_\ell - \rho_q\|_{L^p} \leq  \ell \|\rho_q\|_{C^1} \leq  \ell \lambda_q^\alpha  \leq  \delta_{q+1},
\end{equation}
$$\| u_\ell - u_q\|_{L^{p'}} \leq C \ell \lambda_q^\alpha \leq  \delta_{q+1},$$
$$\| u_\ell - u_q\|_{W^{1,r}} \leq C \ell \lambda_q^\alpha \leq  \delta_{q+1}\, .$$
Indeed note that by \eqref{eqn:choice-beta}
\[
\ell \lambda_q^\alpha = \lambda_q^{-1} = \delta_{q+1}^{\frac{1}{2b\beta}} \ll
\delta_{q+1}
\]
Next observe that
\[
\|\partial_t^N \rho_\ell\|_{C^0} + \| \rho_\ell\|_{C^N} +\| u_\ell\|_{W^{1+N,r}} + \|\partial_t^N u_\ell\|_{W^{1,r}} 
\leq  C(N) \ell^{-N+1}( \| \rho_q\|_{C^1} +\| u_q\|_{W^{2,r}} 
)\leq  C(N) \ell^{-N+1}\lambda_{q}^\alpha
\]
for every $N\in \N\setminus\{0\}$.
Using the Sobolev embedding $W^{d,r}\subset {W^{d,1}\subset }C^0$ we then conclude
\[
\|\partial_t^N u_\ell\|_{C^0} + \|u_\ell\|_{C^N} \leq C (N) \ell^{-N-d+2} \lambda_q^\alpha\, .
\]
By Young's inequality we estimate the higher derivatives of $R_\ell$ in terms of $\|R_q \|_{L^1}$ to get
\begin{equation}\label{e:D_tR}
\| R_\ell\|_{C^N} +\|\partial_t^N R_\ell\|_{C^0} 
\leq \|D^N \rho_\ell \|_{L^\infty} \|R_q \|_{L^1} \leq  C(N) \ell^{-N-d} \leq  C(N) \lambda_{q}^{(1+\alpha)(d+N)}\, 
\end{equation}
for every $N\in \N$. 
Finally, thanks to \cite[Lemma 5.1]{BDLC20} for the last part of the error we have
\begin{equation}\label{e:commutatore}
\|(\rho_q u_q)_\ell - \rho_\ell u_\ell\|_{L^1} \leq C \ell^2 \lambda_q^{2\alpha} \leq \frac{1}{4} \delta_{q+2} ,
\end{equation}
where we have assumed that $a$ is sufficiently large.

\subsection{Definition of the perturbation} \label{sec:perturbation}

Let $\mu_{q+1}>0$ be as in \eqref{eqn:choice-mu} and let $\chi\in C^\infty_c (-\frac{3}{4}, \frac{3}{4})$ such that $\sum_{n \in \mathbb Z} \chi (\tau - n) =1$ for every $\tau\in \mathbb R$.

Fix a parameter $\kappa=\frac{20
}{\delta_{q+2}}$ and consider a finite set $\Lambda \subset \partial B_1 \cap \Q^d $ as in Lemma \ref{lemma:geom} and consider the related building blocks implicitly defined in Lemma \ref{lemma:buildingblocks}.
We define the new density, Hamiltonian ,and vector field by adding to $\rho_\ell$, $H_\ell$ and $u_\ell$ a principal term and a smaller corrector, namely we set
\begin{align*}
\rho_{q+1} &:= \rho_{\ell} + \theta_{q+1}^{(p)}+ \theta_{q+1}^{(c)}\,,\\ 
H_{q+1} &:= H_\ell + h_{q+1}\,,\\
u_{q+1}  &:= u_\ell+  w_{q+1} \  .
\end{align*}
The principal perturbations are given, respectively, by
\begin{align}
\label{eqn:h}
h_{q+1} (x) &= \frac{1}{2^q\lambda_{q+1}} \sum_{\xi \in \Lambda}  H_{\xi , \mu_{q+1}} ( \lambda_{q+1} x),  \\ 
\label{eqn:w}
w_{q+1} (x) &= \frac{1}{2^q} \sum_{\xi \in \Lambda}  X_{H_{\xi , \mu_{q+1}}} ( \lambda_{q+1} x),  \\ 
\label{eqn:theta}
\theta^{(p)}_{q+1} (t,x) &= 2^q \sum_{n \ge 12} \sum_{\xi\in \Lambda } \chi (\kappa |R_\ell (t,x)| - n)   a_{\xi}\left(\frac{R_\ell(t,x)}{|R_\ell(t,x)|}\right) \Theta_{\xi, \mu_{q+1}, n /\kappa} ( \lambda_{q+1} x)\, ,
\end{align}

where we understand that $a_\xi \left(\frac{R_\ell(t,x)}{|R_\ell(t,x)|}\right)$ is well defined because the term $\chi(\kappa |R_\ell(t,x)| -n)$ vanishes at points where $R_\ell$ vanishes.
	Furthermore, in the definition of $\theta^{(p)}_{q+1}$ the first sum runs for $n$ in the range 
	\begin{equation}
	\label{remark:sum is finite}
		12\leq n \leq  C\ell^{-d}\delta_{q+2}^{-1} \le C \lambda_{q}^{d(1+\alpha) + 2\beta b^2}
	\leq C \lambda_{q}^{d(1+\alpha)+1}.
	\end{equation}
	Indeed $\chi(\kappa |R_\ell(t,x)|-n)=0$ if $n\ge 20
	\delta_{q+2}^{-1}\|R_\ell\|_{C^0}+1$ and by \eqref{e:D_tR} we obtain an upper bound for $n$. 
\\
Notice that $u_{q+1}$ is an autonomous Hamiltonian vector field. Indeed, we have that $w_{q+1} = J\nabla h_{q+1}$ and consequently $u_{q+1} = J\nabla H_{q+1}$.\\
The aim of the corrector term for the density is to ensure that the overall addition has zero average:
\[
\theta^{(c)}_{q+1} (t):=- \int \theta_{q+1}^{(p)}(t,x)\, dx.
\]


\section{Proof of the Proposition \ref{prop:inductive}}
In this section, we prove the main iterative proposition \ref{prop:inductive}. The proof is very similar to the one found in \cite{BDLC20} as we have very similar estimates on the building blocks. We provide details for the convenience of the reader.
\begin{lemma}\label{l:uglylemma}
	For $m\in \mathbb N$, $N\in \mathbb N\setminus \{0\}$ and $n\ge 2$ we have
	\begin{align}
	&\|\partial^m_t \chi (\kappa|R_\ell|-n)\|_{C^N} \leq C (m, N) \delta_{q+2}^{-2(N+m)} \ell^{-(N+m)(1+d)} 
	\leq 
	C(m,N) \lambda_q^{(N+m) (d+2) (1+\alpha)}\\
	&\|\partial^m_t (a_\xi ({\textstyle{\frac{R_\ell}{|R_\ell|}}}))\|_{C^N}\leq C (m, N) \delta_{q+2}^{-N-m} \ell^{-(N+m)(1+d)} 
	\leq C(m,N)\lambda_q^{(N+m) (d+2) (1+\alpha)} \qquad\mbox{on $\{\chi (\kappa|R_\ell|-n) >0\}$}.
	\end{align}
\end{lemma}
In the following estimates are crucial the choices of parameters $\ell, b, \beta, \gamma$ and $ \kappa$ fixed in Subsection \ref{sec:param} and the inductive parameters $\lambda_{q+1}, \delta_{q+1}$ and $\mu_{q+1}$.
\subsection{Estimates on the perturbation} Here we show the inductive estimates are satisfied for the perturbed quantities.
\subsubsection{Estimates on the velocity field and Hamiltonian}
\begin{equation}\label{eqn:w}
    \begin{split}
   \|h_{q+1} \|_{W^{1,p'}} \lesssim \|w_{q+1}\|_{L^{p'}} &\leq 2^{-q} |\Lambda| \|X_{H_{\xi, \mu_{q+1}}} (\lambda_{q+1} \cdot)\|_{L^{p'}} \leq 2^{-q} |\Lambda| C \mu_{q+1}^{(d-1)(1/p'-1/p')} \\
    &\leq C |\Lambda| 2^{-q}
    \end{split}
\end{equation}
where we have used \eqref{eqn:Hamiltonian_norms} in the third inequality.
\begin{equation}\label{eqn:w-sobolev}
    \begin{split}
   \|h_{q+1} \|_{W^{2,r}} \lesssim \|w_{q+1}\|_{W^{1,r}} &\leq \|w_{q+1}\|_{L^r} + \|Dw_{q+1}\|_{L^r} \\
    &\leq 2^{-q} |\Lambda| \|X_{H_{\xi, \mu_{q+1}}} (\lambda_{q+1} \cdot)\|_{L^{r}} + 2^{-q} |\Lambda|  \lambda_{q+1} \|D X_{H_{\xi, \mu_{q+1}}} (\lambda_{q+1} \cdot)\|_{L^{r}} \\
    &\leq 2^{-q} |\Lambda| C \mu_{q+1}^{(d-1)(1/p'-1/r)} + 2^{-q} |\Lambda| C \lambda_{q+1} \mu_{q+1}^{1+(d-1)(1/p'-1/r)}\\
    &\leq 2^{-q} |\Lambda| C \lambda_{q+1}^{\gamma(d-1)(1/p'-1/r)} + 2^{-q} |\Lambda| C \lambda_{q+1}^{1 + \gamma(1+(d-1)(1/p'-1/r))}\\
    &\leq 2^{-q} |\Lambda| C \lambda_{q+1}^{-\gamma} + 2^{-q} |\Lambda| C \lambda_{q+1}^{-1/p}\\
    &\leq \delta_{q+1}
    \end{split}
\end{equation}
Here we have used \eqref{eqn:Hamiltonian_norms} in the fourth inequality and \eqref{eqn:choice-mu} in the fifth inequality. 
\subsubsection{Estimates on the density}
Using standard estimates, \eqref{eqn:Thetanorms} and \eqref{remark:sum is finite}
\begin{equation}\label{eqn:theta-pert-p}
\begin{split}
\|\theta^{(p)}_{q+1} (t,\cdot)\|_{L^1} &\leq \sum_{n\ge \nop} \sum_{\xi\in \Lambda} 2^q \| \chi (\kappa |R_\ell (t,x)| - n)  \textstyle{a_{\xi}\left(\frac{R_\ell(t,x)}{|R_\ell(t,x)|}\right)} \|_{C^0} \|\Theta_{\xi, \mu_{q+1}, n /\kappa} (\lambda_{q+1} x) \|_{L^1}\\
&\leq C |\Lambda| 2^q \sum_{n\ge \nop} \frac{n}{\kappa} \mu_{q+1}^{(d-1)(1/p-1)} \leq C |\Lambda| 2^q \kappa^{-1} \lambda_q^{2d(1+\alpha)+2} \mu_{q+1}^{(d-1)(1/p-1)}\\
&\leq C |\Lambda| 2^q \kappa^{-1} \lambda_q^{2d(1+\alpha)+2} \mu_{q+1}^{-(d-1)/p'} \leq \delta_{q+1}
\end{split}
\end{equation}
Now note that by the definition of $\theta^{(c)}_{q+1} (t)$ we have that
\begin{equation}\label{eqn:theta-pert-c}
    |\theta^{(c)}_{q+1} (t)| \leq \|\theta^{(p)}_{q+1} (t,\cdot)\|_{L^1}
\end{equation}
From the above two estimates point $(a)$ of Theorem \ref{prop:inductive} is now proved. Now notice that since $\theta^{(p)}_{q+1}$ is non-negative, we have
\begin{align}
    \inf \theta^{(p)}_{q+1} (t,\cdot) + \theta^{(c)}_{q+1} (t) \geq \theta^{(c)}_{q+1} (t) \geq -\delta_{q+1}
\end{align}
From these computations point $(d)$ of Theorem \ref{prop:inductive} is now proved.
\subsubsection{Estimates for Part (c) of Theorem \ref{prop:inductive}}
Using standard estimates, the improved H\"older inequality \ref{eqn:impr-holder}, \eqref{eqn:Thetanorms} and \eqref{remark:sum is finite} 
\begin{equation}
\label{eqn:tw}
    \begin{split}
        \| \theta_{q+1}^{(p)} w_{q+1} \|_{L^1} &\leq \|\sum_{n \ge 12} \sum_{\xi\in \Lambda } \chi (\kappa |R_\ell (t,x)| - n)   a_{\xi}\left(\frac{R_\ell(t,x)}{|R_\ell(t,x)|}\right) \Theta_{\xi, \mu_{q+1}, n /\kappa} ( \lambda_{q+1} x) X_{H_{\xi,\mu_{q+1}}} (\lambda_{q+1} x) \|_{L^1}\\
        & \le C\lambda_{q+1}^{-1} \sum_{n\ge \nop}\sum_{\xi\in \Lambda} \|  \chi( \kappa|R_\ell| - n ) \textstyle{a_{\xi}\left( \frac{ R_\ell }{ |R_\ell| } \right)} \|_{C^1}
\| \Theta_{\xi, \mu_{q+1}, n/\kappa}X_{H_{\xi,\mu_{q+1}}}\|_{L^1}
\\&\qquad + \sum_{n\ge \nop}\sum_{\xi\in \Lambda} \| \chi( \kappa|R_\ell| - n ) \textstyle{a_{\xi}\left( \frac{ R_\ell }{ |R_\ell| } \right)} \|_{L^1} \| \Theta_{\xi, \mu_{q+1}, n/\kappa}X_{H_{\xi,\mu_{q+1}}}\|_{L^1}
\\& \le C\lambda_{q+1}^{-1} \lambda_{q}^{3(1+\alpha)(d+2)}
+ C \sum_{n\ge \nop} \frac{n}{\kappa} \| \chi( \kappa|R_\ell| - n ) \textstyle{a_{\xi}\left( \frac{ R_\ell }{ |R_\ell| } \right)} \|_{L^1} 
\\& \le \frac{1}{2}\delta_{q+1} + C \|R_\ell\|_{L^1}\le M \delta_{q+1}.
\end{split}
\end{equation}
We could also prove the above indirectly as we know the product cancels with $R_\ell$ up to a very small error.
\subsection{Estimates on higher derivatives}\label{sec:highderiv} Here we show that the perturbed quantities satisfy the estimate \eqref{eqn:ie-2}. By the choice of $\alpha$, since in particular $\alpha\ge 2+\gamma(d+1)$, we have that
\begin{equation}
\label{eqn:est-rhoC1}
\begin{split}
\| \rho_{q+1}\|_{C^1} &\leq \| \rho_{\ell}\|_{C^1}+ \| \theta_{q+1}\|_{C^1} \leq \| \rho_{q}\|_{C^1}
+ 2^q \sum_{n\ge \nop} \sum_{\xi\in \Lambda^{[n]}} \| \chi(\kappa|R_\ell| - n ) \textstyle{ a_{\xi} \left( \frac{ R_\ell }{|R_\ell|}\right)}\|_{C^1} \|\Theta_{\xi, \mu_{q+1},n/\kappa}(\lambda_{q+1}x) \|_{C^1}
\\&\leq C \lambda_q^\alpha + 
C 2^q \lambda_q^{3(1+\alpha)(d+2)} \lambda_{q+1}\mu_{q+1}^{1+(d-1)/p}
\leq \lambda_{q+1}^{\alpha}.
\end{split}
\end{equation}
where in the above calculation, we have used \eqref{eqn:Thetanorms}, \eqref{eqn:choice-mu}, \eqref{remark:sum is finite} and Lemma \ref{l:uglylemma}. An entirely similar estimate is valid for $\|\partial_t \rho_{q+1}\|_{C^0}$. Now we estimate the velocity field using \eqref{eqn:Hamiltonian_norms} and \eqref{eqn:choice-mu}:
\begin{equation}
\label{eqn:est-HamC1}
\begin{split}
\| u_{q+1}\|_{C^0} &\leq \| u_{\ell}\|_{C^0}+ \| w_{q+1}\|_{C^0} \leq \| u_{q}\|_{C^0}
+ 2^{-q} \sum_{\xi\in \Lambda}  \|X_{H_{\xi, \mu_{q+1}}}(\lambda_{q+1}x) \|_{C^0}
\\&\leq \lambda_q^\alpha + 
C 2^{-q} \mu_{q+1}^{1+(d-1)/p'}
\leq \lambda_{q+1}^{\alpha}.
\end{split}
\end{equation}
Now we estimate the Sobolev norms of $u_{q+1}$ which give us the estimates on $H_{q+1}$.
\begin{equation}
    \begin{split}
        \|u_\ell+ w_{q+1}\|_{W^{1,p'}} &\leq \|u_\ell\|_{W^{1,p'}} + \|w_{q+1}\|_{L^{p'}} + \|Dw_{q+1} (t,\cdot)\|_{L^{p'}} \\
        &\leq C\lambda_q^\alpha + 2^{-q}|\Lambda|C + 2^{-q}|\Lambda|C \|D X_{H_{\xi,\mu_{q+1}}}(\lambda_{q+1} \cdot)\|_{L^{p'}}\\
        &\leq C\lambda_q^\alpha + 2^{-q}|\Lambda|C + 2^{-q}|\Lambda|C \lambda_{q+1} \mu_{q+1}^{1+(d-1)(1/p'-1/p')}\\
        &\leq C\lambda_q^\alpha + 2^{-q}|\Lambda|C + 2^{-q}|\Lambda|C \lambda_{q+1}^{1+\gamma} \leq \lambda_{q+1}^\alpha\\
    \end{split}
\end{equation}
where we have used \eqref{eqn:Hamiltonian_norms} and \eqref{eqn:choice-mu}. Similarly we get
\begin{equation}
    \begin{split}
        \|u_\ell+ w_{q+1}\|_{W^{2,r}} &\leq \|u_\ell\|_{W^{2,r}} + \|w_{q+1}\|_{W^{1,r}} + \|D^2w^{(p)}_{q+1} (t,\cdot)\|_{L^r} \\
    &\leq \|u_\ell\|_{W^{2,r}} + 2^{-q} |\Lambda| C + 2^{-q} |\Lambda| C \lambda_{q+1}^2 \|D^2 X_{H_{\xi, \mu_{q+1}}}\|_{L^{r}} \\
    &\leq \|u_\ell\|_{W^{2,r}} + 2^{-q} |\Lambda| C + 2^{-q}|\Lambda| C \lambda_{q+1}^2 \mu_{q+1}^{2+(d-1)(1/p'-1/r)}\\
    &\leq C\lambda_q^\alpha + 2^{-q} |\Lambda| C + 2^{-q} |\Lambda|C \lambda_{q+1}^{2 + \gamma(2+(d-1)(1/p'-1/r))}\\
    &\leq C\lambda_q^\alpha + 2^{-q} |\Lambda| C + 2^{-q} |\Lambda|C \lambda_{q+1}^{\gamma + 1/p'} \leq \lambda_{q+1}^\alpha
    \end{split}
\end{equation}
where we have used \eqref{eqn:w-sobolev} in the second inequality.

\subsection{Estimates on the new Reynold's Stress}
\begin{equation}\label{eqn:new-error}
\begin{split}
-\diver R_{q+1} = & \partial_t \rho_{q+1} + \diver(  \rho_{q+1} u_{q+1}) = \diver ( \theta_{q+1}^{(p)} w_{q+1} - R_\ell) + \partial_t \theta_{q+1}^{(p)}+\partial_t \theta_{q+1}^{(c)} 
\\&+ \diver( \theta_{q+1}^{(p)} u_\ell+ \rho_\ell w_{q+1}  +\theta_{q+1}^{(c)}w_{q+1}
) + \diver ( (\rho_q u_q)_\ell - \rho_\ell u_\ell).
\end{split}
\end{equation}
We firstly observe that point (e) is a consequence of standard mollification properties and the definition of the perturbations in Subsection \ref{sec:perturbation}.

Using Lemma \ref{lemma:geom}, the property $\sum_{n \in \mathbb Z} \chi (\tau - n) =1$ we get
\begin{align}
    \theta_{q+1}^{(p)} w_{q+1} - R_\ell &= \sum_{n \ge 12} \sum_{\xi\in \Lambda } \chi (\kappa |R_\ell (t,x)| - n)   a_{\xi}\left(\frac{R_\ell(t,x)}{|R_\ell(t,x)|}\right) \Theta_{\xi, \mu_{q+1}, n /\kappa} ( \lambda_{q+1} x) X_{H_{\xi,\mu_{q+1}}} (\lambda_{q+1} x) - R_\ell \\
    &= \sum_{n \ge 12} \sum_{\xi\in \Lambda } \chi (\kappa |R_\ell (t,x)| - n)   a_{\xi}\left(\frac{R_\ell(t,x)}{|R_\ell(t,x)|}\right) \left(\Theta_{\xi, \mu_{q+1}, n /\kappa} ( \lambda_{q+1} x) X_{H_{\xi,\mu_{q+1}}} (\lambda_{q+1} x) - \frac{n}{\kappa} \xi \right)\\
    &\quad + \sum_{n \ge 12} \sum_{\xi\in \Lambda } \chi (\kappa |R_\ell (t,x)| - n)   \frac{n}{\kappa} a_\xi \left (\frac{R_\ell(t,x)}{|R_\ell(t,x)|} \right) \xi  -R_\ell
\end{align}
and the last sum is ready to ``cancel'' the error $R_\ell$.
Thus, on taking the divergence, we get
\begin{equation}
    \begin{split}
        \diver(\theta_{q+1}^{(p)} w_{q+1} &- R_\ell)\\
        &= \diver \left(\sum_{n \ge 12} \sum_{\xi\in \Lambda } \chi (\kappa |R_\ell (t,x)| - n)   a_{\xi}\left(\frac{R_\ell(t,x)}{|R_\ell(t,x)|}\right) \Theta_{\xi, \mu_{q+1}, n /\kappa} ( \lambda_{q+1} x) X_{H_{\xi,\mu_{q+1}}} (\lambda_{q+1} x) - R_\ell \right) \\
    &= \sum_{n \ge 12} \sum_{\xi\in \Lambda } \nabla \left(\chi (\kappa |R_\ell (t,x)| - n)   a_{\xi}\left(\frac{R_\ell(t,x)}{|R_\ell(t,x)|}\right)\right) \cdot  \left(\Theta_{\xi, \mu_{q+1}, n /\kappa} ( \lambda_{q+1} x) X_{H_{\xi,\mu_{q+1}}} (\lambda_{q+1} x) - \frac{n}{\kappa} \xi \right) \\
    &\quad + \sum_{n \ge 12} \sum_{\xi\in \Lambda } \diver \left(\chi (\kappa |R_\ell (t,x)| - n)   \frac{n}{\kappa} a_\xi \left ( \frac{R_\ell(t,x)}{|R_\ell(t,x)|} \right) \xi \right) - \diver R_\ell
    \end{split}
\end{equation}
For the first term, we apply the convex integration.
The second term 
where
\[
\tilde R_\ell:= 
\sum_{n \ge 12} \sum_{\xi\in \Lambda } \left(\chi (\kappa |R_\ell (t,x)| - n)   \frac{n}{\kappa} a_\xi \left ( \frac{R_\ell(t,x)}{|R_\ell(t,x)|} \right) \xi \right) = \sum_{n \ge 12} \chi(\kappa |R_\ell|- n)	\frac{R_\ell}{|R_\ell|} \frac{n}{k}.
\]
We have
\begin{align}\label{eqn:tildeRestimate}
\notag
|R_\ell-\tilde R_\ell| & \le \Big|\sum_{n=-1}^{11
} \chi(\kappa |R_\ell|- n) R_\ell \Big|
+\Big| \sum_{n\ge \nop} \chi(\kappa |R_\ell|- n)\left( \frac{R_\ell}{|R_\ell|}\frac{n}{k}-R_\ell\right)	\Big|   
\\& \le \frac{13
}{\kappa} + \sum_{n\ge \nop} \chi(\kappa |R_\ell|- n)\left| |R_\ell|-\frac{n}{\kappa}\right|
\\& \le \frac{13
}{20
}\delta_{q+2}+\frac{3}{40
}\delta_{q+2}
\le \frac{15
}{20
} \delta_{q+2}.
\end{align}
We can now define $R_{q+1}$ which satisfies \eqref{eqn:new-error} as 
\begin{equation}
\begin{split}
-R_{q+1} :=& R^{quadr}+ (\tilde R_\ell- R_\ell) + R^{time}  +\theta_{q+1}^{(p)} u_\ell+ \rho_\ell w_{q+1} +[(\rho_q u_q)_\ell - \rho_\ell u_\ell],
\end{split}
\end{equation}
where
\begin{equation}
\label{defn:Rquadr}
R^{quadr}: = \sum_{n\ge \nop} \sum_{\xi\in \Lambda} \mathcal{R} \left[  \nabla \left(\chi(\kappa |R_\ell|- n)
\textstyle{a_{\xi}\left(\frac{R_\ell}{|R_\ell|}\right)}\right)\cdot  \left((\Theta_{\xi, \mu_{q+1}, n /\kappa} X_{H_{\xi, \mu_{q+1}}})(\lambda_{q+1}x) - \frac{n}{\kappa} \xi \right)\right],
\end{equation}
\begin{equation}
\label{eqn:Rc}
R^{time}:=\nabla \Delta^{-1} (\partial_t \theta_{q+1}^{(p)} +\partial_t \theta_{q+1}^{(c)}),
\end{equation}

Notice that $R^{quadr}$ is well defined and by \eqref{eqn:rightaverage} the function $(\Theta_{\xi, \mu_{q+1}, n /\kappa} X_{H_{\xi, \mu_{q+1}}})(\lambda_{q+1}x)- \frac{n}{\kappa} \xi$ has $0$ mean.
We have that $ \partial_t \theta_{q+1}^{(p)} +\partial_t \theta_{q+1}^{(c)}$ has $0$ mean, so that $R^{time}$ is well defined.\\
We now estimate in $L^1$ each term in the definition of $R_{q+1}$.
Recall that  the estimate on $\|(\rho_q u_q)_\ell - \rho_\ell u_\ell\|_{L^1}$ has been already established in \eqref{e:commutatore}. 

By the property \eqref{ts:antidiv} of the anti-divergence operator $\RR$, Lemma \ref{l:uglylemma} and \eqref{remark:sum is finite} we have
\begin{align*}
\|R^{quadr}\|_{L^1}  & \leq \frac C {\lambda_{q+1}} \sum_{n\ge \nop}\sum_{\xi\in \Lambda}
\| \chi(\kappa |R_\ell|- n)
\textstyle{a_{\xi}\left(\frac{R_\ell}{|R_\ell|}\right)}\|_{C^2}\|\Theta_{\xi, \mu_{q+1}, n/\kappa} X_{H_{\xi, \mu_{q+1}}}(\lambda_{q+1}x)\|_{L^1} \\&\leq C \delta_{q+2} \frac{\lambda_q^{4(1+\alpha)(d+2)+2}}{\lambda_{q+1}} \leq \frac{\delta_{q+2}}{20}.
\end{align*}
To estimate the terms which are linear with respect to the fast variables, we take advantage of the concentration parameter $\mu_{q+1}$. First of all, by Calderon-Zygmund estimates we get
\begin{align*}
\| R^{time} \|_{L^1} \le  C \| \partial_t \theta_{q+1}^{(p)}+\partial_t \theta_{q+1}^{(c)}\|_{L^1}
\le  \| \partial_t \theta_{q+1}^{(p)} \|_{L^1}+ | \partial_t \theta_{q+1}^{(c)} |	.
\end{align*}
Next, notice that
\begin{align}\label{eqn:tpl1}
\| \partial_t \theta_{q+1}^{(p)} \|_{L^1}
& \leq 
C 2^q \sum_{n\ge \nop}\sum_{\xi\in \Lambda}\| \partial_t\big[ \chi(\kappa |R_\ell|- n) \textstyle{a_{\xi}\left( \frac{R_\ell}{|R_\ell|} \right)}\big]\|_{C^0}
\|\Theta_{\xi, \mu_{q+1}, n /\kappa} \|_{L^1}
\\& \leq C 2^q \delta_{q+2} \lambda_{q}^{3(1+\alpha)(d+2)}\mu_{q+1}^{-d/p'} \leq \frac{\delta_{q+2}}{20}.
\end{align}
From \eqref{eqn:tpl1}, \eqref{eqn:itsolves} and \eqref{eqn:l26} we get
\begin{align*}
| \partial_t \theta_{q+1}^{(c)} |	
\le \| \partial_t \theta_{q+1}^{(p)} \|_{L^1}
\le \frac{\delta_{q+2}}{20}
\end{align*} 
Similarly, we have that
\begin{equation} \label{eqn:slow_interaction}
\begin{split}
\| & \theta_{q+1}^{(p)} u_\ell + \rho_\ell w_{q+1}\|_{L^1} \leq \| \theta_{q+1}^{(p)}\|_{L^1} \| u_\ell\|_{L^\infty} + \|\rho_\ell\|_{L^\infty} \|w_{q+1}\|_{L^1}\\
&\leq  \sum_{n\ge \nop} \sum_{\xi \in \Lambda^{[n]}} 2^q \| \chi(\kappa|R_\ell| - n ) \textstyle{ a_{\xi} \left( \frac{ R_\ell }{|R_\ell|} \right) }\|_{L^\infty} \| \Theta_{\xi, \mu_{q+1},n/\kappa}\|_{L^1} \| u_\ell\|_{L^\infty} + \|\rho_\ell\|_{L^\infty} \|X_{H_{\xi, \mu_{q+1}}}\|_{L^1}\\
&\leq C 2^q \delta_{q+2}^{1/p} \lambda_{q}^{2(1+\alpha)(d+2)} \mu^{-d/p'}_{q+1} + C\delta_{q+2}^{1/p'}
\lambda_{q}^{2(1+\alpha)(d+2)}\mu^{-d/p}_{q+1} \leq \frac{\delta_{q+2}}{20}
\end{split}
\end{equation}
In the last inequality we used  $2\beta b^2\le 1$, the definition of $\gamma$, and  $b(1+1/p)\ge 2(1+\alpha)(d+2)+1$.
\subsection{Proof of items (b), (b+) and (c)}
Firstly we prove point (b). By definitions \eqref{eqn:choice-ell}, \eqref{eqn:h} and by the estimates \eqref{eqn:ie-2}, \eqref{eqn:w-sobolev} we have
\begin{align*}
\| H_{q+1} - H_{q} \|_{W^{2,r}} \leq \| H_{q} - H_{\ell} \|_{W^{2,r}} + \| h_{q+1} \|_{W^{2,r}} \leq
\ell \lambda_q^\alpha + \frac{C|\Lambda|}{\lambda_{q+1}^{1/p}}
\leq  \frac{1}{\lambda_q} + \frac{C|\Lambda|}{\lambda_{q+1}^{1/p}} \leq \frac{M}{\lambda_q},
\end{align*}
and similarly we can estimate $\| H_{q+1} - H_{q} \|_{W^{1,p'}}$ using \eqref{eqn:w} instead of \eqref{eqn:w-sobolev}. 
If $p' \geq d-1$, by Lemma \ref{lemma:buildingblocks}, the definition of $\ell$ and \eqref{eqn:ie-2} we get
\begin{align*}
\| H_{q+1} - H_{q} \|_{L^\infty} \leq \| H_{q} - H_{\ell} \|_{L^\infty} + \| h_{q+1} \|_{L^\infty} \leq \frac{1}{\lambda_q} + \frac{C|\Lambda|}{2^q \lambda_{q+1}} \leq \frac{M}{\lambda_q},
\end{align*}
which leads to point (b+).
For point (c), using \eqref{eqn:tw}, \eqref{eqn:slow_interaction}, \eqref{eqn:choice-ell}, $\diver u_\ell =0$ and standard mollification estimates (see for instance \cite[Lemma 5.1]{BDLC20}) we have  the following
\begin{align*}
    \| \rho_{q+1} J \nabla H_{q+1} - \rho_q J \nabla H_q \|_{L^1} &\leq \| (\rho_{\ell} + \theta_{q+1})(u_\ell + w_{q+1}) - \rho_q J \nabla H_q \|_{L^1}\\
    &\leq \| \theta_{q+1}w_{q+1} \|_{L^1} + \| \theta_{q+1} u_\ell + \rho_\ell w_{q+1}\|_{L^1} + \|\rho_{\ell} u_\ell - \rho_q J \nabla H_q \|_{L^1}\\
    &\leq 2\delta_{q+1} + \frac{\delta_{q+2}}{20} + \|\rho_{\ell} u_\ell - (\rho_q u_q)_\ell \|_{L^1} + \|(\rho_q u_q)_\ell - \rho_q J \nabla H_q \|_{L^1}\\
    &\leq M \delta_{q+1}
\end{align*}
Note that the above estimate would imply that $\rho_q J \nabla H_q \to f$ in $L^1$ for some $f \in L^1([0,T] \times \T^d ; \R^d)$. But as $\rho_q \to \rho$ and $J \nabla H_q \to J \nabla H$ in $L^1$, up to subsequences they converge pointwise a.e. Thus we get that $f = \rho J \nabla H$ a.e. on $[0,T] \times \T^d$.
\section{Proof of the main results}
The proof of Theorem \ref{t_ce} and Theorem \ref{t_main} are quite similar to \cite[Theorem 1.4, Theorem 1.3]{BDLC20} respectively, but we write them here for the convenience of the reader.

\subsection{Proof of Theorem \ref{t_ce}}

Without loss of generality we assume $T=1$.
Let $\alpha,b, a_0, M>5$, $\beta>0$ be fixed as in Proposition~\ref{prop:inductive}. Let $a \geq a_0$ be chosen such that
$$ \sum_{q=0}^\infty \delta_{q+1} \leq \frac 1 {32M}.$$
Let $\chi_0$ be a smooth time cut-off which equals $1$ in $[0,1/3]$ and $0$ in $[2/3,1]$, 

We set $ \lambda = 20a$ and define the starting triple $(\rho_0,  H_0, R_0)$ of the iteration as follows:
\[
\rho_0 = \chi_0(t) + \Big(1+\frac{\sin( \lambda x_1)}{4}\Big) (1- \chi_0(t)), \qquad H_0 = 0, \qquad R_0 =- \partial_t \chi_0\frac{ \cos (\lambda x_1)}{4 \lambda} e_1\, .
\]
Simple computations show that the tripe enjoys \eqref{eqn:CE-R} with $q=0$. Moreover $\|R_0\|_{L^1} \leq C \lambda^{-1} = (1/20)C \lambda_0^{-1}$ and thus \eqref{eqn:ie-1} is satisfied because $2\beta <1$ (again we need to assume $a_0$ sufficiently large to absorb the constant). Next $\|\partial_t \rho_0\|_{C^0}+\|\rho_0\|_{C^1}\leq C \lambda = 20C \lambda_0$. Since $H_0\equiv 0$ and $\alpha >1$, we conclude that \eqref{eqn:ie-2} is satisfied as well. 

Next use Proposition~\ref{prop:inductive} to build inductively $(\rho_q, H_q, R_q)$ for every $q\geq 1$. The sequence $\{\rho_q\}_{q\in \N}$ is Cauchy in $C (L^{1})$ and we denote by $\rho \in  C ([0,1], L^{1})$ its limit.
Similarly the sequence of autonomous Hamiltonians $\{ H_q\}_{q\in \N}$ is Cauchy in $W^{1,p'}$ and $ W^{2,r}$; hence, we define $H \in   W^{1,p'} \cap W^{2,r}$ as its limit. Moreover, thanks to the property (c) and the fact that the sequences $\rho_q$ and $J \nabla H_q$ (up to subsequences) converge pointwise a.e. we get that $\rho_q J \nabla H_q$ converges in $C(L^1)$ to $\rho J \nabla H$.

Clearly $\rho$ and $J \nabla H$ solve the continuity equation and $\rho$ is non-negative on $\T^d$ by
\[
\inf_{\T^d} \rho \geq \inf \rho_0 + \sum_{q=0}^\infty  \inf (\rho_{q+1} - \rho_q) \geq \frac 34 - \sum_{q=0}^\infty \delta_{q+1} \geq \frac 14\, .
\]
Moreover, $\rho$ does not coincide with the solution which is constantly $1$, because
\[
\|\rho - 1\|_{L^p} \geq \|1-\rho_0\|_{L^p} - \sum_{q=0}^\infty  \|\rho_{q+1}-\rho_q\|_{L^p}  \geq  \frac 1 {16}- M \sum_{q=0}^\infty \delta_{q+1} >0.
\]
Finally, since $\rho_0(t, \cdot) \equiv 1 $ for $t\in [0,1/3]$, point (c) in Proposition~\ref{prop:inductive} ensures that  $\rho (t, \cdot) \equiv 1$ for every $t$ sufficiently close to $0$.

\subsection{Proof of Theorem~\ref{t_main}}
We first recall a general fact: if $u$ is an everywhere defined Borel vector field in $L^1([0,T] \times \T^d; \R^d)$ such that, for a.e. $x \in \T^d$, the integral curve starting from $x$ is unique, then the corresponding continuity equation is well posed in the class of non-negative, $L^1([0,T] \times \T^d)$ solutions for any $L^1$ initial datum such that $\rho u \in L^1([0,T] \times \T^d)$. 

Indeed, Ambrosio's superposition principle (see e.g. \cite[Theorem 3.2]{A08}) guarantees that each non-negative, $L^1([0,T] \times \T^d)$ solution such that $\rho u \in L^1([0,T] \times \T^d)$ is transported by integral curves of the vector field, namely (no matter how the Borel representative is chosen) there is a probability measures $\eta$ on the space of absolutely continuous curves, supported on the integral curves of the vector field in the sense of Definition~\ref{defn:int-curve},  such that $\rho(t,x)\, \mathscr{L}^d = (e_t)_\# \eta$ for a.e. $t\in [0,T]$ (where $e_t$ is the evaluation map at time $t$). 
Let us consider the disintegration $\{\eta_x\}_{x \in \T^d}$ of $\eta$ with respect to the map $e_0$, which is $\rho_0$-a.e. well defined; since by assumption for a.e. $x \in \T^d$, the integral curve starting from $x$ is unique (and hence coincides with the regular Lagrangian flow),
 we deduce that $\eta_x$ is a Dirac delta on the curve $t \to X(t,x)$ and consequently $\rho(t,\cdot ) \mathscr{L}^d = X(t,\cdot)_\# (\rho_0 \mathscr{L}^d)$. This concludes the proof of the claim.

Let $u= J \nabla H$ be the autonomous Hamiltonian vector field given by Theorem \ref{t_ce} and observe that the Cauchy problem for the continuity equation \eqref{d_continuity} from the initial datum $\rho_0 \equiv 1$ has two different non-negative solutions in $[0,T]$: $\rho^{(1)} \equiv 1$ and the non-constant solution $\rho^{(2)}$ given by Theorem \ref{t_ce}. Hence, by the previous observation we conclude that there exists a set of initial data of positive measure such that the corresponding integral curves are non-unique. Since the fact that the two functions are distinct solutions of the continuity equation is independent of the pointwise representative chosen for the vector field, this completes the proof of Theorem \ref{t_main}.

\section{Non conservation of the Hamiltonian along the trajectories}

\subsection{Proof of Theorem \ref{thm:hamilnotcons}}
We first prove a Theorem in the spirit of Theorem \ref{t_ce}. We fix $p' \geq d-1$.
Let $\alpha,b, a_0, M>5$, $\beta>0$ be fixed as in Proposition~\ref{prop:inductive} and $a \geq a_0$ be chosen such that
$$ \Delta: = \sum_{q=0}^\infty \delta_{q+1} \leq \frac 1 {16}.$$
We fix $ \overline{\psi}_0 \in C^\infty_c ((0,1))$ such that  $\overline{\psi}_0 \geq \Delta$,
$\overline{\psi}_0(x) = \Delta$ for any $x \in [0,1/2]$, $\| \overline{\psi}_0 \|_{L^\infty} \leq 4$, $\int_0^1 \overline{\psi}_0 =1$ (just use the convolution of a proper function and the standard convolution properties).
Then we fix $\overline{H}_0 \in C^\infty_c((0,1/2))$ $\overline{H}_0 \geq 0$ $\int \overline{H}_0 =1,$ $\| \overline{H}_0 \|_{L^\infty} \leq 4 $. We extend these two functions as periodic functions on $\T^d$, imposing that they are independent on the last $d-1$ variables (we call them with the same name with a slight abuse of notation).   
 
Let $\chi_0$ be a smooth time cut-off which equals $1$ in $[0,1/3]$ and $0$ in $[2/3,1]$, 

We set $ \lambda = 20a$ and define the starting triple $(\rho_0,  H_0, R_0)$ of the iteration as follows:
\[
{\rho}_0(t,x) = \chi_0(t) + \overline{\psi}(\lambda x) (1- \chi_0(t)), \qquad {H}_0(x) = \overline{H}_0(\lambda x_1), \qquad R_0 =- \partial_t \chi_0 \RR (\overline{\psi} (\lambda \cdot ) -1) ,
\]
where we have periodized the functions $\overline{\psi}$ and $\overline{H}_0$ with the parameter $\lambda_0$.
Simple computations show that the tripe enjoys \eqref{eqn:CE-R} with $q=0$. Moreover, thanks to Lemma \ref{lemma23}, $\|R_0\|_{L^1} \leq C \lambda^{-1} = (1/20)C \lambda_0^{-1}$ and thus \eqref{eqn:ie-1} is satisfied because $2\beta <1$ (again we need to assume $a_0$ sufficiently large to absorb the constant). Next $\|\partial_t \rho_0\|_{C^0}+\|\rho_0\|_{C^1}\leq C \lambda = 20C \lambda_0$. Since  $\alpha >1$, we conclude that \eqref{eqn:ie-2} is satisfied as well. 

Next use Proposition~\ref{prop:inductive} to build inductively $(\rho_q, H_q, R_q)$ for every $q\geq 1$. The sequence $\{\rho_q\}_{q\in \N}$ is Cauchy in $C (L^{1})$ and we denote by $\rho \in  C ([0,1], L^{1})$ its limit.
Similarly the sequence of autonomous Hamiltonians $\{ H_q\}_{q\in \N}$ is Cauchy in $W^{1,p'}$, $ W^{2,r}$ and $L^\infty$ (for the last property we used property (b+) of Proposition \ref{prop:inductive} since $p' >d-1$); hence, we define $H \in  W^{1,p'} \cap W^{2,r} \cap  L^\infty$ as its limit, that is also continuous. 
Moreover, thanks to the property (c) and the fact that the sequences $\rho_q$ and $J \nabla H_q$ converge a.e. we get that $\rho_q J \nabla H_q$ converges in $C(L^1)$ to $\rho J \nabla H$.

Clearly $\rho$ and $J \nabla H$ solve the continuity equation and $\rho$ is non-negative on $\T^d$ by
\[
\inf_{\T^d} \rho \geq \inf \rho_0 + \sum_{q=0}^\infty  \inf (\rho_{q+1} - \rho_q) \geq \Delta - \Delta = 0.
\]

Now we apply the Ambrosio's  superposition principle (see e.g. \cite[Theorem 3.2]{A08}) to the non negative solution $\rho$ (note that $\rho J \nabla H \in L^1$), which guarantees that $\rho$
is transported by integral curves of the vector field $J \nabla H$, namely (no matter how the Borel representative is chosen) there is a probability measures $\eta$ on the space of absolutely continuous curves, supported on the integral curves of the vector field in the sense of Definition~\ref{defn:int-curve},  such that $\rho(t,x)\, \mathscr{L}^d = (e_t)_\# \eta$ for a.e. $t\in [0,T]$ (where $e_t$ is the evaluation map at time $t$). 
We use the notation $\{\eta_x\}_{x \in \T^d}$, that are probability measure defined by the disintegration of $\eta$ with respect to the map $e_0$, which is $\mathcal{L}^d$-a.e. well defined.
 To conclude the proof of our theorem is sufficient to prove that 
$$ \int_{\T^d} \int_{AC} H(\gamma (0)) d \eta_x (\gamma) dx > \int_{\T^d} \int_{AC} H(\gamma (1)) d \eta_x (\gamma) dx,$$
because $\eta_x$ is concentrated, for $\mathcal{L}^d$ a.e. $x \in \T^d$, on the family of absolutely continuous integral curves of $J \nabla H$. Thanks to the superposition principle it is equivalent to prove that 
$$ \int_{\T^d}  H(x) \rho(0,x) dx >  \int_{\T^d}  H(x) \rho(1,x) dx,$$
notice that the solution $\rho$ of the continuity equation with respect to the vector field $J \nabla H$ is independent to the pointwise representative of $J \nabla H$, this would conclude the proof.

 By properties (b+), (d) of Proposition \ref{prop:inductive} and the definition of $\rho_0$, 
we have the following estimates
\begin{align*}
\int_{\T^d}  H(x) \rho(0,x) dx = \int_{\T^d}  H(x) dx \geq  \int_{\T^d}  H_0(x) dx - \| H - H_0 \|_{L^\infty} \geq 1 - \frac{1}{\lambda_0}.
\end{align*}
and
\begin{align*}
\int_{\T^d} &  H(x) \rho(1,x) dx 
\\
&= \int_{\T^d}  (H(x) - H_0(x)) \rho(1,x) dx + \int_{\T^d}  H_0(x)( \rho(1,x) - \rho_0(1,x)) dx + \int_{\T^d}  H_0(x) \rho_0(1,x) dx.
\end{align*}
By property (b+), the definition of $\lambda_{q+1}= \lambda_q^b$ and the definition of $H_0$ and $\rho_0$ we estimate every summand and we get
\begin{align*}
 \int_{\T^d}  (H(x) - H_0(x)) \rho(1,x) dx  \leq \| H - H_0 \|_{L^\infty} \| \rho\|_{L^\infty_tL^1_x} \leq \frac{1}{\lambda_0},
 \\
 \int_{\T^d}  H_0(x)( \rho(1,x) - \rho_0(1,x)) dx \leq \| H_0 \|_{L^\infty} \| \rho - \rho_0 \|_{L^\infty_tL^1_x} \leq 4 \Delta,
 \\
  \int_{\T^d}  H_0(x) \rho_0(1,x) dx \leq 4 \Delta.
\end{align*}
The thesis follows observing that $1 > \frac{2}{\lambda_0} + 8 \Delta$.

\subsection*{Acknowledgements}
The authors would like to thank Thomas Alazard, Camillo De Lellis and Maria Colombo for introducing them to the problem. The authors also thank Elio Marconi and Riccardo Tione for advice on the introduction. VG has been supported by the National Science Foundation under Grant No. DMS-FRG-1854344 and MS has been supported by the SNSF Grant 182565.

\bibliography{bibliografia}
\bibliographystyle{alpha}

\end{document}